\pdfoutput=1
\documentclass[journal]{new-aiaa}

\usepackage{amsthm,graphicx,algorithm,algpseudocode,amsmath,amsfonts,verbatim,mathtools,bm,hyperref,aircraftshapes,enumitem}

\usepackage{longtable,tabularx}
\title{Real-Time Quadrotor Trajectory Optimization with Time-Triggered Corridor Constraints }

\author{Yue Yu \footnote{Postdoctoral research fellow, Oden Institute of Computational Sciences and Engineering; yueyu@utexas.edu
.} and Kartik Nagpal\footnote{Research assistant, Oden Institute of Computational Sciences and Engineering; kartiknagpal@utexas.edu
.}}
\affil{The University of Texas at Austin, Austin, TX 78712}
\author{Skye Mceowen \footnote{Ph.D. Candidate, The Department of Aeronautics and Astronautics; skye95@uw.edu.} and Beh\c{c}et~A\c{c}\i kme\c{s}e\footnote{Professor, The Department of Aeronautics and Astronautics; behcet@uw.edu.}}
\affil{University of Washington, Seattle, WA 98195}
\author{Ufuk Topcu\footnote{Associate Professor, Oden Institute of Computational Sciences and Engineering; utopcu@utexas.edu}}
\affil{The University of Texas at Austin, Austin, TX 78712.}

\usepackage[style=base]{caption}
\usepackage{subcaption}

\usetikzlibrary{calc,patterns,positioning}

\allowdisplaybreaks

\newcommand{\diag}{\mathop{\rm diag}}

\newcommand{\norm}[1]{\left\lVert#1\right\rVert}
\newcommand{\mnorm}[1]{{\left\vert\kern-0.25ex\left\vert\kern-0.25ex\left\vert #1 
    \right\vert\kern-0.25ex\right\vert\kern-0.25ex\right\vert}}

\newtheorem{remark}{Remark}
\newtheorem{proposition}{Proposition}
\newtheorem{assumption}{Assumption}

\newcommand{\ie}{{\it i.e.}}

\begin{document}

\maketitle

\begin{abstract}
One of the keys to flying quadrotors is to optimize their trajectories within the set of collision-free corridors. These corridors impose nonconvex constraints on the trajectories, making real-time trajectory optimization challenging. We introduce a novel numerical method that approximates the nonconvex corridor constraints with time-triggered convex corridor constraints. This method combines bisection search and repeated infeasibility detection. We further develop a customized C++ implementation of the proposed method, based on a first-order conic optimization method that detects infeasibility and exploits problem structure. We demonstrate the efficiency and effectiveness of the proposed method using numerical simulation on randomly generated problem instances as well as indoor flight experiments with hoop obstacles. Compared with mixed integer programming, the proposed method is about 50--200 times faster.  
\end{abstract}

\section*{Nomenclature}

{\renewcommand\arraystretch{1.0}
\noindent\begin{longtable*}{@{}l @{\quad=\quad} l@{}}
\multicolumn{2}{@{}l}{Sets}\\
\(\mathbb{N}\)  & the set of positive integers \\
\(\mathbb{R}, \mathbb{R}_+\) &  the set of real and non-negative real numbers \\
\(\mathbb{H}_i\) & the set of feasible position vectors in the \(i\)-th corridor \\
\(\mathbb{V}\) & the set of feasible velocity vectors\\
\(\mathbb{U}_a\) & the set of thrust vectors with pointing direction and magnitude upper bound constraints\\
\(\mathbb{U}_b\) & the set of thrust vectors with magnitude lower bound\\
\(\mathbb{W}\) & the set of feasible thrust rate\\

\multicolumn{2}{@{}l}{Parameters}\\
\(m\) & quadrotor mass\\
\(g\) & acceleration vector caused by gravity\\
\(\Delta\) & sampling time period\\
\(\omega\) & weighting parameter for thrust rates\\
\(c_i, d_i, \rho_i, \eta_i\) & center coordinates, direction vector, radius, and length of the \(i\)-th cylindrical corridor\\
\(\xi\) & maximum speed\\
\(\underline{\gamma}, \overline{\gamma}, \theta\) & minimum thrust magnitude, maximum thrust magnitude, and maximum tilting angle\\
\(\delta\) & maximum thrust rate\\
\(\overline{r}_0, \overline{v}_0\) & initial position and initial velocity of the quadrotor\\
\(\overline{r}_f, \overline{v}_f, \overline{u}_f\) & final position, final velocity, and final thrust of the quadrotor\\
\multicolumn{2}{@{}l}{Variables}\\
\(r_k, v_k, u_k\) & position, velocity, thrust of the quadrotor at time \(k\Delta\)\\
\(\tau_i\) & the length of the trajectory segment for the \(i\)-th corridor\\
\(t\) & total length of trajectory\\
\(b_{ik}\) & binary variable, takes value \(1\) if the quadrotor is in the \(i\)-th corridor at time \(k\Delta\)
\end{longtable*}}

\section{Introduction}
\label{sec: intro}

One of the keys to flying quadrotors in a dynamically changing environment is to optimize their trajectories subject to dynamics and collision-avoidance constraints in real-time \cite{elmokadem2021towards,lan2021survey}. Along such a trajectory, the position of the quadrotor needs to stay within a set of \emph{collision-free corridors}. Each corridor is a bounded convex flight space; the union of all these corridors form a nonconvex pathway connecting the quadrotor's current position to its target position \cite{ioan2019obstacle,ioan2020navigation}; see Fig~\ref{fig: corridor} for a simple illustration. To avoid collisions with obstacles whose positions change rapidly or uncertain, it is critical to update these corridors in real-time. As a result, one needs to optimize trajectories subject to nonconvex corridor constraints in real-time: the faster the optimization, the faster the quadrotor can react to sudden changes of the obstacles. 

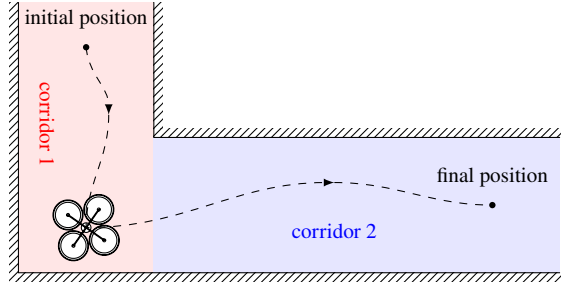
\begin{figure}[!ht]
	\centering
	\begin{tikzpicture}[scale=0.6]
		
		\coordinate (O1) at (0, 0);
		\coordinate (O2) at (0, -3);
		\coordinate (O3) at (12, -3);
		\coordinate (O4) at (12, 0);
		\coordinate (O5) at (3, 0);
		\coordinate (O6) at (3, 3);
		\coordinate (O7) at (0, 3);
		\coordinate (O8) at (3, -3);
		\coordinate (x0) at (1.5, 2);
		\coordinate (xf) at (10.5, -1.5);
		\coordinate (quad) at (1.5, -2);
		\coordinate (mid1) at (2, 0.5);
		\coordinate (mid2) at (7, -1);
		
		\fill[red!10] (O2) -- (O8) -- (O6) --(O7) --cycle;
		\node[label={[red, rotate=-90]below:\footnotesize corridor 1}] at (1, 0.5) {};
		
		\fill[blue!10] (O8) -- (O3) -- (O4) --(O5) --cycle;
		\node[label={[blue]below:\footnotesize corridor 2}] at (7, -1.5) {};
		
		\draw (O7) -- (O2) -- (O3);
		\draw (O4) -- (O5) -- (O6);
		
		\pattern[pattern=north east lines] ($(O7)+(-0.2, 0)$) rectangle ($(O2)+(0, -0.2)$);
		
		\pattern[pattern=north east lines] ($(O2)+(0, -0.2)$) rectangle (O3);
		
		\pattern[pattern=north east lines] ($(O4)+(0, 0.2)$) rectangle ($(O5)+(0, 0)$);
		
		\pattern[pattern=north east lines] ($(O5)+(0.2, 0.2)$) rectangle (O6);
		
		\fill (x0) circle [radius=2pt];
		\node[label=above:{\footnotesize initial position}] at (x0) {};
		
 		\fill (xf) circle [radius=2pt];
 		\node[label=above:{\footnotesize final position}] at (xf) {};
		
		\node [quadcopter top,fill=white,draw=black,minimum width=1cm,rotate=10,scale=0.8] at (quad) {}; 
		
		\draw[dashed, -latex] (x0) to[out=-85,in=90] (mid1);
		
		\draw[dashed, -latex] (mid1) to[out=-90,in=90] (quad) to[out=0,in=180] (mid2);
		
		\draw[dashed] (mid2) to[out=0,in=180] (xf);

       	\end{tikzpicture}
		\caption{A quadrotor flight trajectory with corridor constraints.}
		\label{fig: corridor}
\end{figure}

Since the flight space defined by the union of the corridors is nonconvex, optimizing the trajectories for the quadrotor is computationally challenging. One standard solution approach is \emph{mixed integer programming} \cite{grossmann2002review,richards2005mixed,ioan2020mixed}, which first uses binary variables to describe the union of all corridors, then optimizes quadrotor trajectories together with these binary variables \cite{richards2002aircraft,mellinger2012mixed,tang2015mixed,landry2016aggressive}. However, the worst-case computation time of this approach increases exponentially as the number of binary variables increases. As a result, even with the state-of-the-art solvers--such as GUROBI \cite{gurobi}--real-time quadrotor trajectory optimization via mixed integer programming is still difficult, if at all possible. Alternatively, one can model the corridor constraints as smooth nonconvex constraints and solve the resulting trajectory optimization using the successive convexification method \cite{mao2018successive}. But this approach suffers from slow computation speed \cite{szmuk2019real}, and requires careful parameter tuning to ensure the desired algorithm convergence \cite{malyuta2021convex}. 

Recently, there has been an increasing interest in approximating the nonconvex corridor constraints with time-triggered constraints, where each convex corridor is activated only within one time interval \cite{mellinger2011minimum,yu2014energy,deits2015efficient,watterson2015safe,liu2016high,janevcek2017optiplan,liu2017planning,mohta2018fast,gao2018optimal}. These approximations make the resulting trajectory optimization convex and thus computationally more tractable. However, the existing results have the following limitations. First, they only consider polytopic constraints on trajectory variables, such as elementwise upper and lower bounds on the velocity and acceleration of the quadrotor. These polytopic constraints do not accurately capture the geometric structure of many practical operational constraints--such as the magnitude and pointing direction constraint of the thrust vector \cite{szmuk2017convexification,szmuk2018real,szmuk2019real}--and flight corridors with nonpolytopic boundaries--such as cylindrical or spherical corridors. Second, to our best knowledge, none of the existing methods explicitly test whether the resulting trajectory optimization is feasible. Consequently, the resulting trajectory optimization can be close to infeasible, in which case, a numerical solver will fail to provide a solution; or the trajectory optimization can be far away from being infeasible, which can cause conservative trajectories with unnecessarily long time of flight.  

We introduce a novel bisection method that approximates the nonconvex corridor constraints using time-triggered convex corridor constraints, and develop customized implementation of this method that enables real-time quadrotor trajectory optimization subject to general second-order constraints. Our contributions are as follows. 
\begin{enumerate}
    \item We theoretically prove that nonconvex corridor constraints are equivalent to time-varying convex corridor constraints, provided that an optimal triggering time for each corridor is known.
    \item We propose a novel bisection method to estimate the optimal triggering time via repeated infeasibility detection in conic optimization. This method systematically reduces the trajectory length while ensuring that the resulting trajectory optimization is feasible up to a given tolerance. The estimated triggering time reduces a nonconvex trajectory optimization problem to a sequence of convex ones.
    \item We develop a customized C++ trajectory optimization solver based on the bisection method. This solver automatically detects infeasibility and exploits the sparsity and geometric structure of trajectory optimization by implementing the proportional-integral projected gradient method (PIPG), an efficient first-order primal-dual conic optimization method.
    \item We demonstrate the application of the proposed bisection method using numerical simulation and indoor flight experiments. Compared with mixed integer programming, the proposed bisection method and C++ solver shows 50--200 times speedups at the price of an increase in the cost function value by less than 10\% .
\end{enumerate}

The implications of our work are threefold. First, our work sets a new benchmark for real-time quadrotor trajectory optimization, which significantly improves the mixed integer programming approach in terms of computation time. Second, our work provides a fresh perspective to deal with nonconvexity in collision avoidance for general autonomous vehicles using bisection search and infeasibility detection. Third, our work demonstrates the potential of PIPG--and in general, first-order optimization methods--in solving nonconvex optimal control problems via not only numerical simulation but also flight experiments.    

\paragraph*{Notation} Given a real number \(\alpha\in\mathbb{R}\), we let \(\lfloor \alpha\rfloor\) denote the largest integer lower bound of \(\alpha\), and \(\lceil\alpha \rceil\) denote the smallest integer upper bound of \(\alpha\). Given a vector \(r\) and a matrix \(M\), we let \(\norm{r}\) denote the \(\ell_2\)-norm of vector \(r\), \([r]_j\) denote the \(j\)-th element of vector \(r\), and \(\mnorm{M}\) denote the largest singular value of matrix \(M\). We let \(1_n\) and \(0_n\) denote the \(n\)-dimensional vector whose entries are all 1's and all 0's, respectively. We let \(0_{m\times n}\) denote the \(m\times n\) zero matrix, and \(I_n\) denote the \(n\times n\) identity matrix. Given a closed convex cone \(\mathbb{K}\), we let \(\mathbb{K}^\circ\) denote its polar cone. Given \(i, j\in\mathbb{N}\) with \(i< j\) and \(a_i, a_{i+1}, \ldots, a_{j-1}, a_j\in\mathbb{R}^n\), we let \(a_{[i, j]}\coloneqq \begin{bmatrix}
a_i^\top & a_{i+1}^\top & \ldots & a_{j-1}^\top & a_j^\top
\end{bmatrix}^\top\). We say an constrained optimization is \emph{feasible} if its constraints can be satisfied, and \emph{infeasible} otherwise.

\section{Three-degree-of-freedom dynamics model for quadrotors}
\label{subsec: quad model}

Trajectory optimization for a dynamical system requires a mathematical model that predicts the future state of the system given its current state and input. We introduce a quadrotor dynamics model with three-degrees-of-freedom (3DoF), along with various constraints on the position, velocity, thrust, and thrust rate of the quadrotor. This model lays the foundation of the trajectory optimization problem in the next section. 

\subsection{Three degree-of-freedom dynamics}

We consider a 3DoF dynamics model for a quadrotor. In particular, at time \(s\in\mathbb{R}_+\), we let \(r(s)\in\mathbb{R}^3\) and \(v(s)\in\mathbb{R}^3\) denote the position and velocity of the center of mass of the quadrotor, and \(u(s)\in\mathbb{R}^3\) denote the total thrust force provided by the propellers. Furthermore, we let \(m\in\mathbb{R}_+\) and \(g=\begin{bmatrix}0 & 0 & -9.81\end{bmatrix}^\top\) denote the mass of the quadrotor and the acceleration vector caused by gravity, respectively. The 3DoF continuous-time dynamics model for quadrotor dynamics is described by the following set of differential equations:
\begin{equation}\label{sys: CT}
\begin{aligned}
\frac{d}{ds}r(s)&=v(s),\\ \frac{d}{ds}v(s)&=\frac{1}{m}u(s)+g.
\end{aligned}
\end{equation}

We discretize the above continuous-time differential equation using a first-order-hold scheme. Particularly, we let \(\Delta\in\mathbb{R}_+\) denote the discretization step size. Let
\begin{equation}
   r_k\coloneqq r(k\Delta),\enskip v_k\coloneqq v(k\Delta),\enskip u_k\coloneqq u(k\Delta),  
\end{equation} 
for all \(k\in\mathbb{N}\).
We apply a piecewise linear input thrust such within each \(\Delta\in\mathbb{R}_+\) time interval, \ie, 
\begin{equation}
     u(s)=\left(k+1-\frac{s}{\Delta}\right)u_k+\left(\frac{s}{\Delta}-k\right)u_{k+1},
\end{equation}
for all \(k\Delta\leq s\leq (k+1)\Delta\). Under this assumption, the equations in \eqref{sys: CT} are equivalent to the following:
\begin{equation}\label{sys: DT}
    \begin{aligned}
    r_{k+1}=&  r_k+\Delta v_k+\frac{\Delta^2}{3 m}\left(u_k+\frac{1}{2}u_{k+1}\right)+\frac{\Delta^2}{2}g,\\
    v_{k+1}=&  v_k+\frac{\Delta}{2m}(u_k+u_{k+1})+\Delta g,   
    \end{aligned}
\end{equation}
for all \(k\in\mathbb{N}\).

\subsection{Position, velocity, and thrust constraints}

The position, velocity, and thrust vector of the quadrotor are subject to the following constraints.

\subsubsection{Position}
The quadrotor's position is constrained within the union of a set of three-dimensional cylinders, or \emph{corridors}. We let \(l\in\mathbb{N}\) denote the total number of corridors. For the \(i\)-th corridor, we let \(c_i\in\mathbb{R}^3\) denote its center, \(d_i\in\mathbb{R}^3\) with \(\norm{d_i}=1\) denote its direction vector, \(\eta_i\in\mathbb{R}_+\) and \(\rho_i\in\mathbb{R}_+\) denote its half-length and radius, respectively. See Fig.~\ref{fig: cylinder corridor} for an illustration.  We define the \(i\)-th corridor as follows:
\begin{equation}\label{eqn: cylinder}
    \mathbb{H}_i\coloneqq \{c_i+r\in\mathbb{R}^3|\norm{r-\langle d_i, r\rangle d_i}\leq \rho_i, |\langle d_i, r\rangle|\leq \eta_i\}.
\end{equation}

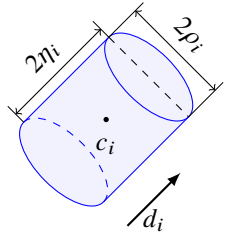
\begin{figure}[!ht]
    \centering
    \begin{tikzpicture}[scale=1]
    \coordinate (o) at (0, 0);
    \coordinate (a) at (45:0.8);
    \coordinate (b) at (225:0.8);
    \coordinate (a1) at ($(a)+(135:0.75)$);
    \coordinate (a2) at ($(a)+(-45:0.75)$);
    \coordinate (b1) at ($(b)+(135:0.75)$);
    \coordinate (b2) at ($(b)+(-45:0.75)$);
    \coordinate (b3) at ($(b)+(-45:1.2)$);
    
    \draw [blue,fill=blue!5,rotate around={-45:(0,0)}] (a1) -- (b1) arc (180:360:0.75 and 0.35) -- (a2) arc (0:180:0.75 and 0.35);
    
    \draw[blue,dashed,rotate around={-45:(0,0)}] (b2) arc (0:180:0.75 and 0.35);
    \draw[blue,rotate around={-45:(0,0)}] (a1) arc (180:360:0.75 and 0.35);
    
    \fill (o) circle [radius=1pt];
    \node[label=below:{ $c_i$}] at (o) {};
    
    \draw[very thin] (a1) -- ($(a1)+(135:0.3)$);
    \draw[very thin] (b1) -- ($(b1)+(135:0.3)$);
    
    \draw[<->] ($(a1)+(135:0.15)$) --  ($(b1)+(135:0.15)$) node [rotate=45, midway, above] { $2\eta_i$};
    
    \draw[very thin] (a1) -- ($(a1)+(45:0.65)$);
    \draw[very thin] (a2) -- ($(a2)+(45:0.65)$);
    \draw[dashed] (a1) -- (a2);
    
    \draw[<->] ($(a1)+(45:0.5)$) --  ($(a2)+(45:0.5)$) node [rotate=-45, midway, above] { $2\rho_i$};
    
    \draw[-latex, thick] (b3) -- ($(b3) + (45: 1)$) node [midway, below] { $d_i$};
    
\end{tikzpicture}
    \caption{A illustration of the set \(\mathbb{H}_i\) in \eqref{eqn: cylinder}.}
    \label{fig: cylinder corridor}
\end{figure}

\subsubsection{Velocity}
The quadrotor's speed is upper bounded by \(\xi\in\mathbb{R}_+\). The set of feasible velocity vectors is as follows:
\begin{equation}\label{eqn: v ball}
    \mathbb{V}\coloneqq \{v\in\mathbb{R}^3| \norm{v}\leq \xi\}.
\end{equation}

\subsubsection{Thrust}
The thrust vectors of the quadrotor are subject to the following two different set of constraints: magnitude constraints and direction constraints.
\paragraph{Magnitude constraints} The Euclidean norm of the thrust vector is upper bounded by \(\overline{\gamma}\in\mathbb{R}_+\), and the thrust along the direction opposite to the gravity is lower bounded by \(\underline{\gamma}\in\mathbb{R}_+\).
\paragraph{Direction constraints} The direction of the thrust vector is constrained as follows: the angle between the thrust direction and the the direction opposite to the gravity is no more than a fixed angle \(\theta\in[0, \frac{\pi}{2}]\).

The above constraints on the thrust magnitude and direction ensure that the on-board motors can provide the thrust needed, and the tilting angle of the quadrotor is upper bounded. See Fig.~\ref{fig: tilting} for an illustration of the tilting angle.
\begin{figure}[!ht]
    \centering
    \begin{tikzpicture}
    \coordinate (a) at (0, 0);
    
    \node [quadcopter side,fill=white,draw=black,minimum width=3.5cm,rotate=30,scale=1] at (a) {}; 
    
    \draw[red, -latex] (a) -- node[right] {\small gravity} ($(a)+(0, -1.2)$) ;
    
    \draw[dashed] (a) -- ($(a)+(90:1.3)$);
    \draw[blue, -latex] (a) -- node[left] {\small thrust} ($(a)+(120:1.5)$);
    
    \draw[latex-] ($(a)+(90:1)$) arc (90:105:1) node[above]{$\theta$};
    \draw[-latex] ($(a)+(105:1)$) arc (105:120:1);
    
    \fill (a) circle [radius=1pt] {}; 
     
\end{tikzpicture}
    \caption{Tilting angle \(\theta\) of the quadrotor.}
    \label{fig: tilting}
\end{figure}
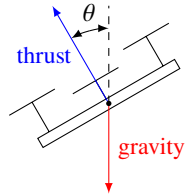
By combining the aforementioned constraints, we define the set of feasible thrust vectors as the intersection of the following two sets:
\begin{subequations}\label{eqn: u muffin}
\begin{align}
    \mathbb{U}_a \coloneqq & \{u\in\mathbb{R}^3| \norm{u}\leq \overline{\gamma}, \cos\theta \norm{u}\leq [u]_3\},\label{eqn: u icecream}\\
    \mathbb{U}_b \coloneqq & \{u\in\mathbb{R}^3| [u]_3\geq \underline{\gamma}\}.\label{eqn: u lower bound}
\end{align}
\end{subequations}

\subsubsection{Thrust rate} The difference between two consecutive thrust vectors, termed a \emph{thrust rate vector}, is subject to an upper bound of \(\delta\in\mathbb{R}_+\) on its Euclidean norm. The set of all feasible thrust rate vectors is as follows:  
\begin{equation}\label{eqn: w ball}
    \mathbb{W}\coloneqq \{w\in\mathbb{R}^3| \norm{w}\leq \delta\}.
\end{equation}
The constraints in \eqref{eqn: w ball} prevents large changes in the thrust vector within a \(\Delta\)-seconds time interval, hence ensuring the smoothness of the thrust trajectory. 

\section{Trajectory optimization with time-triggered constraints}
\label{sec: trigger}

We will introduce the quadrotor trajectory optimization with time-triggered corridor constraints. To this end, we will first consider the trajectory optimization with nonconvex corridor constraints, then propose an approximate problem that replaces these nonconvex corridor constraints with convex ones. 

\subsection{Trajectory optimization with nonconvex corridor constraints}

We will introduce a trajectory optimization problem subject to nonconvex corridor constraints. In this problem, we use the quadrotor dynamics in \eqref{sys: DT}. We let \(t\in\mathbb{N}\) denote the total length of the trajectory. We let \(\overline{r}_0, \overline{v}_0\in\mathbb{R}^3\) denote the known initial position and initial velocity of the quadrotor, respectively. Similarly, we let \(\overline{r}_f, \overline{v}_f, \overline{u}_f\in\mathbb{R}^3\)
denote the known final position, final velocity, and final thrust of the quadrotor, respectively. We will use the set \(\mathbb{V}\), \(\mathbb{U}\) and \(\mathbb{W}\) defined in \eqref{eqn: v ball}, \eqref{eqn: u muffin}, and \eqref{eqn: w ball}, respectively. We let \(\{\mathbb{H}_1, \mathbb{H}_2, \ldots, \mathbb{H}_l\}\) denote a sequence of corridors, where \(\mathbb{H}_i\) is defined by \eqref{eqn: cylinder} for all \(i\in[1, l]\). 

We now introduce the following quadrotor trajectory optimization with nonconvex state constraints, where \(\omega\in\mathbb{R}_+\) is a weight scalar for the cost for the thrust rates: by changing the value of \(\omega\), one can obtain different trade-offs between cost for the thrust and the thrust rates.

\vspace{1em}
\noindent\fbox{%
\centering
\parbox{0.96\linewidth}{%
{\bf{Trajectory optimization with nonconvex corridor constraints}}
\begin{equation}\label{opt: nonconvex}
    \begin{array}{ll}
        &  \underset{u_{[0, t]}}{\mbox{minimize}} \enskip \frac{1}{2}\sum_{k=0}^{t} \norm{u_k}^2+\frac{\omega}{2}\sum_{k=0}^{t-1} \norm{u_{k+1}-u_k}^2\\
        &\mbox{subject to} \\ &r_{k+1}=r_k+\Delta v_k+\frac{\Delta^2}{3 m}(u_k+\frac{1}{2}u_{k+1})+\frac{\Delta^2}{2}g,\\
        &v_{k+1}=v_k+\frac{\Delta}{2m}(u_k+u_{k+1})+\Delta g, \enskip \forall k\in[0, t-1],\\
       & u_{k+1}-u_k\in\mathbb{W}, \enskip \forall k\in[0, t-1],\\
        &u_k\in\mathbb{U}_a\cap \mathbb{U}_b,\enskip v_k\in\mathbb{V}, \enskip r_k\in\bigcup_{i=1}^{l}\mathbb{H}_i,\enskip \forall k\in[0, t],\\
        & r_0=\overline{r}_0, \enskip v_0=\overline{v}_0,\enskip  r_t=\overline{r}_f, \enskip v_t=\overline{v}_f, \enskip u_t=\overline{u}_f.
    \end{array}
\end{equation}
}%
}
\vspace{1em}

Optimization \eqref{opt: nonconvex} is equivalent to a mixed integer optimization problem. To see this equivalence, notice that optimization~\eqref{opt: nonconvex} contains the following constraints:
\begin{equation}\label{eqn: noncvx corridor}
    r_k\in\bigcup\limits_{i=1}^{l}\mathbb{H}_i, \enskip \forall k\in[0, t].
\end{equation}
The constraints in \eqref{eqn: noncvx corridor} are equivalent to the following set of constraints with binary variables:
\begin{subequations}\label{eqn: mixed-integer corridor}
\begin{align}
&\norm{(r_k-c_i)-\langle d_i, r_k-c_i\rangle}\leq b_{ik}\rho_i+\mu(1-b_{ik}),\label{eqn: big M cylinder 1}\\
& |\langle d_i, r_k-c_i\rangle|\leq b_{ik}\eta_i+\mu(1-b_{ik}),\label{eqn: big M cylinder 2}\\
& b_{ik}\in\{0, 1\}, \enskip \sum_{i=1}^{l} b_{ik}\geq 1, \enskip  \forall k\in[0, t], i\in[1, l],\label{eqn: binary}
\end{align}
\end{subequations}
where \(\mu\in\mathbb{R}_+\) denotes a very large positive scalar. Indeed, if \(b_{ik}=0\), then the constraints in \eqref{eqn: big M cylinder 1} and \eqref{eqn: big M cylinder 2} become redundant, since \(\mu\) is very large. On the other hand, if \(b_{ik}=1\), then the constraints in \eqref{eqn: big M cylinder 1} and \eqref{eqn: big M cylinder 2} imply that \(r_k\in\mathbb{H}_i\). Finally, the constraints in \eqref{eqn: binary} in \eqref{eqn: binary} ensure that there exists \(i\in[1, l]\) such that \(b_{ik}=1\), hence \(r_k\in\mathbb{H}_i\) for some \(i\in[1, l]\). Therefore, the constraints in \eqref{eqn: noncvx corridor} and \eqref{eqn: mixed-integer corridor} are equivalent. 

Since optimization \eqref{opt: nonconvex} is equivalent to a mixed-integer optimization, the computation time for solving optimization \eqref{opt: nonconvex} increases exponentially as the number of integer variables--in this case, jointly determined by the trajectory length \(\tau\) and number of corridors \(l\)--increases. Consequently, a real-time solution method is only possible if the values of \(\tau\) and \(n\) are both sufficiently small. 

\subsection{Trajectory optimization with time-triggered corridor constraints}

We will show that optimization~\eqref{opt: nonconvex} takes a simpler form if we know \emph{a priori} the sequence of corridors that the optimal trajectory traverses. To this end, we start with the following assumption on the ordering of corridor sequence \(\{\mathbb{H}_1, \mathbb{H}_2, \ldots, \mathbb{H}_l\}\).

\begin{assumption}\label{asp: trigger}
Suppose optimization~\eqref{opt: nonconvex} is feasible. Let \(u_{[0, t]}\) be an optimal thrust trajectory for optimization \eqref{opt: nonconvex}, and \(r_{[0, t]}\) and \(u_{[0, t]}\) satisfy the constraints in optimization \eqref{opt: nonconvex}. There exists \(\tau_1, \tau_2, \ldots, \tau_l\in\mathbb{R}_{+}\) such that \(t=\sum_{j=1}^l\tau_j\) and \(r_k\in\mathbb{H}_i\) for all \(k\in[\sum_{j=1}^{i-1}\tau_j, \sum_{j=1}^i\tau_j]\) and \(i\in[1, l]\), where \(\sum_{j=1}^0\tau_j\coloneqq 0\).
\end{assumption}

For Assumption~\ref{asp: trigger} to hold, we need to know \emph{a priori} the ordered sequence of corridors that order in which the optimal trajectory traverses. Many corridor generating algorithms, such as convex lifting, can provide such an ordered sequence of corridors; see \cite{ioan2019obstacle,ioan2020navigation} for some recent examples.

Assumption~\ref{asp: trigger} also implies that no corridor appeared more than once along the optimal corridor path. Since reentering the same corridor twice will increase the value of the objective function in optimization~\eqref{opt: nonconvex}, such an implication always holds in practice.

Under Assumption~\ref{asp: trigger}, it is tempting to replace the nonconvex corridor constraints in \eqref{eqn: noncvx corridor} with time-varying constraints. After this replacement, optimization~\eqref{opt: nonconvex} becomes the following optimization in \eqref{opt: trigger}.

\vspace{1em}
\noindent\fbox{%
\centering
\parbox{0.96\linewidth}{
{\bf{Trajectory optimization with time-triggered corridor constraints}}
\begin{equation}\label{opt: trigger}
    \begin{array}{ll}
        &  \underset{u_{[0, t]}}{\mbox{minimize}} \enskip \frac{1}{2}\sum_{k=0}^t \norm{u_k}^2+\frac{\omega}{2}\sum_{k=0}^{t-1} \norm{u_{k+1}-u_k}^2\\
        &\mbox{subject to} \\ &r_{k+1}=r_k+\Delta v_k+\frac{\Delta^2}{3 m}(u_k+\frac{1}{2}u_{k+1})+\frac{\Delta^2}{2}g,\\
        &v_{k+1}=v_k+\frac{\Delta}{2m}(u_k+u_{k+1})+\Delta g, \enskip \forall k\in[0, t-1],\\
        & u_{k+1}-u_k\in\mathbb{W}, \enskip \forall k\in[0, t-1],\\
        &u_k\in\mathbb{U}_a\cap\mathbb{U}_b,\enskip v_k\in\mathbb{V},\enskip \forall k\in[0, t],\\
        & r_k\in\mathbb{H}_i, \enskip \forall k\in[\sum_{j=1}^{i-1}\tau_j, \sum_{j=1}^i\tau_j], \enskip i\in[1, l],\\
        & r_0=\overline{r}_0, \enskip v_0=\overline{v}_0,\enskip  r_t=\overline{r}_f, \enskip v_t=\overline{v}_f, \enskip u_t=\overline{u}_f.
    \end{array}
\end{equation}
}%
}
\vspace{1em}

The following proposition shows that, under Assumption~\ref{asp: trigger}, solving optimization~\eqref{opt: trigger} is equivalent to solving optimization~\eqref{opt: nonconvex}.

\begin{proposition}\label{prop: trigger}
Suppose that Assumption~\ref{asp: trigger} holds. If \(u^\star_{[0, t]}\) is an optimal thrust trajectory for optimization~\eqref{opt: trigger}, then \(u^\star_{[0, t]}\) is an an optimal thrust trajectory for optimization~\eqref{opt: nonconvex}. 
\end{proposition}

\begin{proof}
Since Assumption~\eqref{asp: trigger} holds, optimization~\eqref{opt: nonconvex} has at least one optimal solution, and so does optimization~\eqref{opt: trigger}. Let \(u^\star_{[0, t]}\) be an optimal solution for optimization~\eqref{opt: trigger}, \(u_{[0, t]}\) be an optimal solution for optimization \eqref{opt: nonconvex}, \(\phi(u_{[0, t]})=\frac{1}{2}\sum_{k=0}^t \norm{u_k}^2+\frac{\omega}{2}\sum_{k=0}^{t-1} \norm{u_{k+1}-u_k}^2\), and \(\phi(u^\star_{[0, t]})=\frac{1}{2}\sum_{k=0}^t \norm{u_k^\star}^2+\frac{\omega}{2}\sum_{k=0}^{t-1} \norm{u_{k+1}^\star-u_k^\star}^2\).

First, since trajectory \(u_{[0, t]}\) also satisfies the constraints in \eqref{opt: trigger} and \(u^\star_{[0, t]}\) is optimal for optimization~\eqref{opt: trigger}, we must have
\(\phi(u^\star_{[0, t]})\leq \phi(u_{[0, t]})\).

Second, Assumption~\ref{asp: trigger} implies that \(u^\star_{[0, t]}\) also satisfies the constraints in optimization~\eqref{opt: nonconvex}. Combining this fact with the assumption that \(u_{[0, t]}\) is optimal for optimization~\eqref{opt: nonconvex}, we conclude that
\(\phi(u_{[0, t]})\leq \phi(u^\star_{[0, t]})\).

Therefore we conclude that \(u^\star_{[0, t]}\) satisfies the constraints in optimization~\eqref{opt: nonconvex} and \(\phi(u^\star_{[0, t]})=\phi(u_{[0, t]})\). Hence \(u^\star_{[0, t]}\) is also optimal for optimization~\eqref{opt: nonconvex}. 
\end{proof}

Proposition~\ref{prop: trigger} provides valuable insights in solving optimization~\eqref{opt: nonconvex}: rather than the value of the \((t+1)l\) binary variables in \eqref{eqn: mixed-integer corridor}, we only need to determine the value of \(l\) integers that determines the triggering time, given by \(\tau_1, \tau_2, \ldots, \tau_l\), in optimization~\eqref{opt: trigger}. Although computing the exact value of this sequence is as difficult as solving optimization~\eqref{opt: nonconvex} itself, one can compute a good approximation very efficiently, as we will show next.

\subsection{Computing the triggering time via bisection method}
\label{subsec: min time}

In this section, we introduce a numerical algorithm for optimization \eqref{opt: trigger} using an approximate triggering time sequence \(\tau_{[1, l]}\). To this end, we make the following assumption about optimization~\eqref{opt: trigger}.

\begin{assumption}\label{asp: bound}
There exists \(\underline{\tau}_1, \underline{\tau}_2, \ldots, \underline{\tau}_l\) and \(\overline{\tau}_1, \overline{\tau}_2, \ldots, \overline{\tau}_l\) with \(\underline{\tau}_j\leq \overline{\tau}_j\) for all \(j=1, 2, \ldots, l\), such that 1) optimization~\eqref{opt: trigger} is feasible if \(\tau_{[1, l]}=\overline{\tau}_{[0, l]}\) and \(t=\sum_{j=1}^l\overline{\tau}_j\), and 2) optimization~\eqref{opt: trigger} is infeasible if \(\tau_{[1, l]}=\underline{\tau}_{[0, l]}\) and \(t=\sum_{j=1}^l\underline{\tau}_j\).
\end{assumption}

\begin{remark}
Assumption~\ref{asp: trigger} implies that optimization~\eqref{opt: trigger} is feasible if we allocate an sufficient amount of time for each corridor, and infeasible otherwise. Using the length of each corridor and an upper and lower bounds on the average speed of the quadrotor, we can obtain an interval estimate for each corridor.
\end{remark}

Given lower and upper bound sequences that satisfy Assumption~\ref{asp: trigger}, we introduce a heuristic method, summarized in Algorithm~\ref{alg: trigger}. The idea is to first use a bisection search method to tighten the interval bounds for each corridor, one at a time. Then using these tightened upper bounds to solve optimization~\eqref{opt: trigger}.

\begin{algorithm}[!ht]
\caption{Trajectory optimization with time-triggered corridor constraints}
\begin{algorithmic}[1]
\Require Two time sequence \(\overline{\tau}_{[1, l]}\) and \(\underline{\tau}_{[1, l]}\) that satisfy Assumption~\ref{asp: bound}, positive accuracy tolerance \(\epsilon\). 
\For{\(i=1, 2, \ldots, l\)}\label{alg: trigger start}
\While{\(\overline{\tau}_i-\underline{\tau}_i>\epsilon\)}\label{alg: bisec start}
\State{\(\hat{\tau}_j=\begin{cases}
\lfloor \frac{1}{2}(\overline{\tau}_i+\underline{\tau}_i)\rfloor, & \text{if \(j=i\).}\\
\overline{\tau}_j, & \text{otherwise.}
\end{cases}\)}
\State{Let \(t=\sum_{j=1}^l\hat{\tau}_j\) and \(\tau_{[1, l]}=\hat{\tau}_{[1, l]}\) in optimization~\eqref{opt: trigger}.
}
\If{optimization \eqref{opt: trigger} is infeasible}
\State{\(\underline{\tau}_i\gets \hat{\tau}_i\)}
\Else 
\State{\(\overline{\tau}_i\gets \hat{\tau}_i\)}
\EndIf
\EndWhile\label{alg: bisec end}
\EndFor\label{alg: trigger end}
\State{Let \(\tau_{[1, l]}=\overline{\tau}_{[1, l]}\) and \(t=\sum_{j=1}^l\overline{\tau}_j\) in optimization \eqref{opt: trigger}, then solve for the optimal trajectory \(u^\star_{[0, \tau]}\).}\label{alg: solve trigger}
\Ensure \(u^\star_{[0, t]}\)
\end{algorithmic}
\label{alg: trigger}
\end{algorithm}

We note that the upper bound sequence \(\overline{\tau}_{[0, l]}\) computed by the for-loop between line~\ref{alg: trigger start} and line~\ref{alg: trigger end} in  Algorithm~\ref{alg: trigger} is not necessarily the same sequence in Assumption~\ref{asp: trigger}. Consequently the instance of optimization~\eqref{opt: trigger} solved in line~\ref{alg: solve trigger} is merely an \emph{approximation} of optimization~\eqref{opt: nonconvex}. However, such an approximation has the following attractive properties. First a feasible solution is guaranteed to exist by construction, and each convex corridor constraint is active within the corresponding time interval. Second, up to the accuracy tolerance \(\epsilon\), each element of the upper bound sequence is reduced greedily until optimization~\eqref{opt: trigger} becomes infeasible, which reduces the conservativeness of the initial estimates.

\section{Real-time conic optimization with infeasibility detection}
\label{sec: PIPG}
The key step in Algorithm~\ref{alg: trigger} is to solve optimization~\eqref{opt: trigger} if it is feasible, and prove that it is infeasible otherwise. Such a problem is also known as \emph{infeasibility detection} in constrained optimization. In this section we introduce an infeasibility detection method customized for optimization~\eqref{opt: trigger}. This method is based on the proportional-integral projected gradient method (PIPG), a primal-dual conic optimization method \cite{yu2020proportional,yu2021proportionalA,yu2021proportionalB,yu2022extrapolated}. 

\subsection{Reformulating a trajectory optimization as a conic optimization}
\label{subsec: conic}

Conic optimization is the minimization of a convex objective function subject to conic constraints. In the following, we will reformulate the trajectory optimization problem in \eqref{opt: trigger} as a special case of conic optimization. To this end, we need to rewrite the objective function and constraints in optimization~\eqref{opt: trigger} in a more compact form as follows. First, we introduce the following trajectory variable:
\begin{equation}\label{eqn: traj var}
    x\coloneqq \begin{bmatrix}
    r_{[0, t]}^\top & v_{[0, t]}^\top & u_{[0, t]}^\top & w_{[0, t-1]}^\top
    \end{bmatrix}^\top.
\end{equation}
where \(w_k\coloneqq u_{k+1}-u_k\) for all \(k\in[0, t-1]\). With this variable, we can rewrite the quadratic objective function in optimization~\eqref{opt: trigger} as follows:
\begin{equation}\label{eqn: quad obj}
\begin{aligned}
\frac{1}{2}x^\top \underbrace{\diag\left(\begin{bmatrix}
0_{6(t+1)}^\top & 1_{3(t+1)}^\top & \omega 1_{3t}^\top 
\end{bmatrix}^\top \right)}_Px.    
\end{aligned}
\end{equation}
Second, we define the following submatrices: 
\begin{equation}\label{eqn: H blocks}
\begin{aligned}
    &H_{11}=\begin{bmatrix}
    0_{3t\times 3} & I_{3t}
    \end{bmatrix}-\begin{bmatrix}
    I_{3t} & 0_{3t\times 3}
    \end{bmatrix}, \\
    &H_{12}=-\Delta \begin{bmatrix}
    I_{3t} & 0_{3t\times 3} \end{bmatrix},\enskip H_{14}=0_{3t\times 3t},\\
    & \textstyle H_{13}=-\frac{\Delta^2}{3m}\begin{bmatrix}
    0_{3t\times 3} & I_{3t}
    \end{bmatrix}-\frac{\Delta^2}{6m}\begin{bmatrix}
    I_{3t} & 0_{3t\times 3}
    \end{bmatrix},\\
    & H_{21}=0_{3t\times 3(t+1)}, \enskip H_{22}=H_{11},\enskip H_{24}=0_{3t\times 3t},\\
    &\textstyle H_{23}=-\frac{\Delta}{2m}\begin{bmatrix}
    0_{3t\times 3} & I_{3t}
    \end{bmatrix}-\frac{\Delta}{2m}\begin{bmatrix}
    I_{3t} & 0_{3t\times 3}
    \end{bmatrix}, \\
    & H_{31}=H_{32}=0_{3t\times 3(t+1)}, \enskip H_{33}=H_{11},\\
    & H_{34}=I_{3t},\enskip H_{41}=H_{42}=0_{(t+1)\times 3(t+1)}, \\ &H_{43}=I_{t+1}\otimes \begin{bmatrix}
    0 & 0 & 1
    \end{bmatrix}, \enskip H_{44}=0_{(t+1)\times 3t}.
\end{aligned}
\end{equation}
With the definition in \eqref{eqn: traj var} and \eqref{eqn: H blocks}, we can rewrite the linear equality and inequality constraints in optimization~\eqref{opt: trigger}--which include the linear dynamics constraints and the linear lower bound constraints on the thrust vectors--equivalently as follows:
\begin{equation}\label{eqn: linear constr}
    \underbrace{\begin{bmatrix}
    H_{11} & H_{12} & H_{13} & H_{14}\\
    H_{21} & H_{22} & H_{23} & H_{24}\\
    H_{31} & H_{32} & H_{33} & H_{34}\\
    H_{41} & H_{42} & H_{43} & H_{44}\\
    \end{bmatrix}}_{H}x-\underbrace{\begin{bmatrix}
    0_{9t}\\
    \underline{\gamma}1_{t+1}
    \end{bmatrix}}_b\in \underbrace{\{0_{9t}\}\times \mathbb{R}_+^{t+1}}_{\mathbb{K}},
\end{equation}
Note that \(\underline{\gamma}\in\mathbb{R}_+\) is the thrust lower bound introduced in \eqref{eqn: u lower bound}.

Third, we define the following closed convex set
\begin{equation}\label{eqn: D_i}
    \mathbb{D}_i=\mathbb{H}_i\times \mathbb{V}\times\mathbb{U}_a\times \mathbb{W}, 
\end{equation}
for all \(i=1, 2, \ldots, l\), where set \(\mathbb{H}_i\), \(\mathbb{V}\), \(\mathbb{W}\) are given in \eqref{eqn: cylinder}, \eqref{eqn: v ball}, \eqref{eqn: w ball}, respectively; set \(\mathbb{U}_a\) is given by \eqref{eqn: u icecream}.

Notice the only difference between set \(\mathbb{U}_a\cap\mathbb{U}_b\) and set \(\mathbb{U}_a\) is that the latter does not include the linear lower bound constraint in \(\mathbb{U}_b\); this constraint is already included in the last \(t+1\) linear inequality constraints in \eqref{eqn: linear constr}. With these sets, we can compactly rewrite the second-order-cone constraints in optimization~\eqref{opt: trigger}--which include those for position, velocity, thrust, and thrust rate vectors--as follows:
\begin{equation}\label{eqn: SOC constr}
    x\in\underbrace{(\mathbb{D}_1)^{\tau_1}\times (\mathbb{D}_2)^{\tau_2}\times \cdots \times (\mathbb{D}_l)^{\tau_l}}_{\mathbb{D}},
\end{equation}
where \((\mathbb{D}_i)^{\tau_i}\) is the Cartesian product of \(\tau_i\) copies of set \(\mathbb{D}_i\).

With the above definition, we can now rewrite optimization~\eqref{opt: trigger} equivalently as optimization~\eqref{opt: conic}, where matrix \(P\) is given in \eqref{eqn: quad obj}; matrix \(H\), vector \(b\), cone \(\mathbb{K}\) are given in \eqref{eqn: linear constr}; set \(\mathbb{D}\) is given in \eqref{eqn: SOC constr}.

\vspace{1em}
\noindent\fbox{
\centering
\parbox{0.94\linewidth}{
{\bf{Conic optimization}}
\begin{equation}\label{opt: conic}
    \begin{array}{ll}
        \underset{x}{\mbox{minimize}} &  \frac{1}{2} x^\top P x\\
        \mbox{subject to} & Hx-b\in\mathbb{K}, \enskip x\in\mathbb{D}.
    \end{array}
\end{equation}
}
}
\vspace{1em}

\begin{figure}[!ht]
    \centering
    \includegraphics[width=0.35\linewidth]{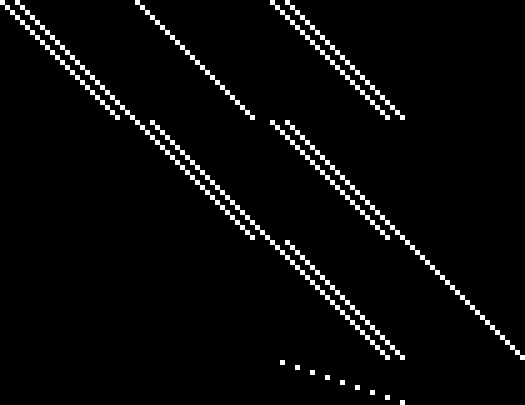}
    \caption{The sparsity pattern of matrix \(H\) in \eqref{eqn: linear constr}. Each zero and nonzero entry corresponds to a black and white pixel, respectively.}
    \label{fig: sparse}
\end{figure}

Optimization~\eqref{opt: conic} has two salient features: the sparsity pattern of matrix \(P\) and \(H\), and the geometric structure of set \(\mathbb{D}\). First, matrix \(P\) is diagonal, and matrix \(H\) has many zero elements; see Fig.~\ref{fig: sparse} for an illustration. The presence of these zero elements is because the dynamics constraints in \eqref{sys: DT} only apply to variables corresponding to adjacent time steps. Second, set \(\mathbb{D}\) is a Cartesian product of many simple sets, such as cylinder, ball, or the intersection of an icecream cone and a ball. See Fig.~\ref{fig: sets} for an illustration.

\begin{figure}[!ht]
    \centering
    \begin{subfigure}[b]{0.49\columnwidth}
    \centering 
    \begin{tikzpicture}[scale=0.7]
    \coordinate (a) at (45:0.8);
    \coordinate (b) at (225:0.8);
    \coordinate (a1) at ($(a)+(135:0.75)$);
    \coordinate (a2) at ($(a)+(-45:0.75)$);
    \coordinate (b1) at ($(b)+(135:0.75)$);
    \coordinate (b2) at ($(b)+(-45:0.75)$);
    
    \draw [blue,fill=blue!5,rotate around={-45:(0,0)}] (a1) -- (b1) arc (180:360:0.75 and 0.35) -- (a2) arc (0:180:0.75 and 0.35);
    
    \draw[blue,dashed,rotate around={-45:(0,0)}] (b2) arc (0:180:0.75 and 0.35);
    \draw[blue,rotate around={-45:(0,0)}] (a1) arc (180:360:0.75 and 0.35);
    
\end{tikzpicture}
    \caption{Cylinder.}
    \label{fig: cylinder}
    \end{subfigure}
    \begin{subfigure}[b]{0.49\columnwidth}
    \centering 
   \begin{tikzpicture}[scale=0.7]
    \coordinate (a) at (0, 0);
    \coordinate (a1) at (-1, 0);
    \coordinate (a2) at (1, 0);
    
    \draw[blue,fill=blue!5] (a) circle (1);
    
    \draw[blue]  (a1) arc (180:360: 1 and 0.2);
    \draw[blue,dashed] (a2) arc (0:180: 1 and 0.2);
    
\end{tikzpicture}
    \caption{Ball.}
    \label{fig: ball}
    \end{subfigure}
    \begin{subfigure}[b]{\columnwidth}
    \centering 
    \begin{tikzpicture}[scale=0.7]

    \coordinate (a) at (0, 0);
    \coordinate (a1) at (60:2);
    \coordinate (a2) at (120:2);
    
    \draw[blue,fill=blue!5] (a) -- (a1) arc (30:150:1.155) -- (a);
    \draw[blue,dashed] (a1) arc (0:180:1 and 0.2);
    \draw[blue] (a2) arc (180:360:1 and 0.2);

\end{tikzpicture}
    \caption{The intersection of an icecream cone and a ball.}
    \label{fig: icecream}
    \end{subfigure}
    \caption{An illustration of the geometric structure of the simple sets that constitute the set \(\mathbb{D}\) in \eqref{eqn: SOC constr}. }
    \label{fig: sets}
\end{figure}
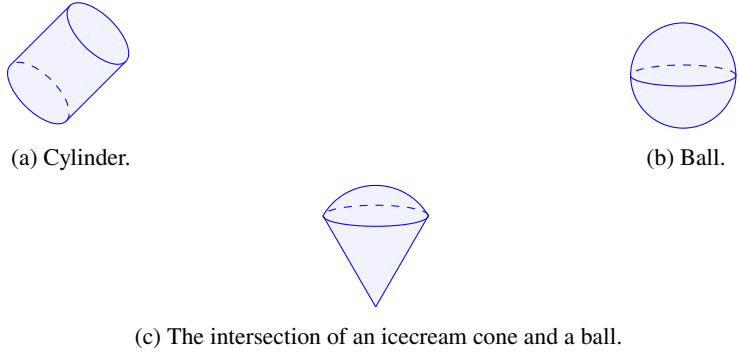

\subsection{Proportional-integral projected gradient method}
\label{subsec: PIPG}

To exploit the salient features of optimization~\eqref{opt: conic}, we propose to use the proportional-integral projected gradient method (PIPG). PIPG is a state-of-the-art first-order primal-dual optimization method that combines the idea of projected gradient method and proportional-integral feedback of constraint violation; such a combination was first introduced in distributed optimization  \cite{yu2020mass,yu2020rlc} and later extended to optimal control problems \cite{yu2020proportional,yu2021proportionalA,yu2021proportionalB,yu2022extrapolated}. 

Algorithm~\ref{alg: PIPG} is the pseudocode implementation of PIPG with extrapolation \cite{yu2022extrapolated}, where \(\Pi_{\mathbb{D}}\) and \(\Pi_{\mathbb{K}}\) denote the Euclidean projection map onto set \(\mathbb{D}\) and the polar cone of cone \(\mathbb{K}\), respectively; these projection maps will be discussed in details later. The if-clause between line~\ref{alg: inf start} and line~\ref{alg: inf end} determines whether optimization~\eqref{opt: conic} is infeasible by monitoring the difference between two consecutive iterates \cite{yu2022extrapolated}. 

\begin{algorithm}[!ht]
\caption{PIPG for convex trajectory optimization}
\begin{algorithmic}[1]
\Require Parameters in optimization~\eqref{opt: conic}, number of iterations \(j_{\max}\), step sizes \(\alpha, \beta, \lambda\), feasibility tolerance \(\epsilon\)\label{alg: pipg start}
\State Randomly initialize \(x,\overline{x}\in\mathbb{R}^{12t+9}\), \(y,\overline{y}\in\mathbb{R}^{10t+1}\). 
\For{\(j=1, 2, \ldots, j_{\max}\)}
\State{\(y^-\gets y\)}
\State{\(x\gets\Pi_{\mathbb{D}}[\overline{x}-\alpha(P\overline{x}+H^\top \overline{y})]\)}\label{alg: proj D}
\State{\(y\gets \Pi_{\mathbb{K}^\circ}[\overline{y}+\beta(H(2x-\overline{x})-g)]\)}\label{alg: proj K polar}
\State{\(\overline{x}\gets(1-\lambda)\overline{x}+\lambda x\)}
\State{\(\overline{y}\gets(1-\lambda)\overline{y}+\lambda y\)}
\EndFor\label{alg: pipg end}
\If{\(\frac{\norm{y-y^-}}{\beta\lambda \norm{x}}\leq \epsilon\)}\label{alg: inf start}
\State{\Return{\(x\)}}
\Else 
\State{\Return{``Infeasible"}}
\EndIf \label{alg: inf end}
\Ensure \(x\) or ``Infeasible".
\end{algorithmic}
\label{alg: PIPG}
\end{algorithm}

Compared with other numerical methods for optimization~\eqref{opt: conic}, PIPG has the following advantages. First, PIPG does not compute the inverse of any matrices or solve any linear equation systems, making it suitable for real-time implementation with light digital footprints \cite{yu2021proportionalB}. Second, compared with other first-order methods, PIPG achieves the fastest convergence rates in terms of both the primal-dual gap and constraint violation \cite{yu2021proportionalA}. Third, PIPG automatically generates proof of infeasibility if possible \cite{yu2021proportionalB,yu2022extrapolated}. When solving optimal control problems, PIPG is much faster than many state-of-the-art optimization solvers in numerical experiments \cite{yu2022extrapolated}.

\subsection{Implementation}
\label{subsec: implementation}

In order to implement Algorithm~\ref{alg: PIPG}, we need to determine several algorithm parameters, and efficiently compute the projections in line~\ref{alg: proj D} and line~\ref{alg: proj K polar} of Algorithm~\ref{alg: PIPG}. We will discuss these implementation details in the following.

\subsubsection{Parameter selection}
\paragraph{Step sizes} The iterates of PIPG converge if parameter \(\alpha\) and \(\beta\) satisfy the following constraint, which is a special case of those in \cite[Rem. 1]{yu2022extrapolated}:
\begin{equation}
     0<\alpha=\beta<\frac{2}{\sqrt{\mnorm{P}^2+4\mnorm{H}^2}}.
\end{equation}
By using the definition of matrix \(P\) in \eqref{eqn: quad obj}, one can verify that \(\mnorm{P}=\max\{1, \omega\}\), where \(\omega\) is the weighting parameter in the objective function in optimization~\eqref{opt: trigger}. As for the value of \(\mnorm{H}^2\), we compute an approximate of it using the \emph{power iteration algorithm} \cite{kuczynski1992estimating}, summarized in Algorithm~\ref{alg: power}. 

\begin{algorithm}[!ht]
\caption{The power iteration method \cite{kuczynski1992estimating} }
\begin{algorithmic}[1]
\Require Matrix \(H\in\mathbb{R}^{(10t+1)\times (12t+9)}\), accuracy tolerance \(\epsilon\).
\State Randomly initialize \(x\in\mathbb{R}^{12t+9}\); let \(\sigma=\norm{x}, \sigma^-=\epsilon+\sigma\). 
\While{\(|\sigma-\sigma^-|\geq \epsilon\)}
\State{\(\sigma^-\gets \sigma\)}
\State{\(y\gets \frac{1}{\sigma}Hx\)}
\State{\(x\gets H^\top y\)}
\State{\(\sigma\gets \norm{x}\)}
\EndWhile
\Ensure \(\sigma\).
\end{algorithmic}
\label{alg: power}
\end{algorithm}

As for parameter \(\lambda\)--which denotes the step length of extrapolation in PIPG \cite{yu2022extrapolated}--  numerical experiments shows that values between \(1.6\) and \(1.9\) leads to the best convergence performance in practice \cite{yu2022extrapolated}. In our implementation, we let \(\lambda=1.9\). 

\paragraph{Maximum number of iteration and feasibility tolerance}
As a first order method, PIPG tends to converge within hundreds of iterations. In the implementation of Algorithm~\ref{alg: PIPG}, we set \(j_{\max}=10^4\) and \(\epsilon=10^{-3}\).

\subsubsection{Computing the projections}

We now provide explicit formulas for computing the projections onto the closed convex sets that constitute the set \(\mathbb{D}\) in \eqref{eqn: SOC constr}; see Fig.~\ref{fig: sets} for an illustration. For projection formulas of other closed convex sets, such as the cone \(\mathbb{K}\) in \eqref{eqn: linear constr}, we refer the interested readers to \cite[Chp. 29]{bauschke2017convex}.

\paragraph{Cylinder} Given a position vector \(r\in\mathbb{R}^3\), the projection of \(r\) onto the set \(\mathbb{H}_i\) in \eqref{eqn: cylinder} is given as follows \cite[Exe. 29.1]{bauschke2017convex}:
\begin{equation}
    \Pi_{\mathbb{H}_i}[r]=\max(-\eta_i, \min(\eta_i, r_a))+\frac{\rho_i}{\max(\norm{r_b}, \rho_i)}r_b,
\end{equation}
where 
\begin{equation}
    r_a=\langle d_i, r\rangle d_i, \enskip r_b = r-\langle d_i, r\rangle d_i.
\end{equation}
\paragraph{Ball} Given a velocity vector \(v\in\mathbb{R}^3\), the projection of \(v\) onto the set \(\mathbb{V}\) in \eqref{eqn: v ball} is given as follows \cite[Prop. 29.10]{bauschke2017convex}:
\begin{equation}\label{eqn: vel ball proj}
    \Pi_{\mathbb{V}}[v]=\frac{\xi}{\max(\xi, \norm{v})}v.
\end{equation}
The projection onto set \eqref{eqn: w ball} is similar.
\paragraph{The intersection of a ball and an icecream cone}
Computing a projection onto the intersection of a icecream cone and a ball is the same as first computing a projection onto the icecream cone then computing a projection onto the ball \cite[Thm. 7.1]{bauschke2018projecting}. In particular, give a thrust vector \(u\in\mathbb{R}^3\), the projection of \(u\) onto the set \(\mathbb{U}_a\) in \eqref{eqn: u icecream} is given by
\begin{equation}\label{eqn: thrust ball proj}
    \Pi_{\mathbb{U}_a}[u]=\frac{\overline{\gamma}}{\max(\overline{\gamma}, \norm{u_a})}u_a,
\end{equation}
where
\begin{equation}\label{eqn: thrust cone proj}
    u_a=\begin{cases}
    u, \quad \text{if } \cos\lambda\norm{u}\leq [u]_3,\\
    0, \quad \text{if } \sin\lambda\norm{u}\leq -[u]_3,\\
    \langle u, u_b\rangle u_b, \quad \text{otherwise,}
    \end{cases}
\end{equation}
and 
\begin{equation}
    u_b=\begin{bmatrix}
    0\\
    0\\
    \cos\lambda
    \end{bmatrix}+\frac{\sin\lambda}{\sqrt{([u]_1)^2+([u]_2)^2}}\begin{bmatrix}
    [u]_1\\
    [u]_2\\
    0
    \end{bmatrix}.
\end{equation}
The formula in \eqref{eqn: thrust ball proj} is similar to that in \eqref{eqn: vel ball proj}. The formula in \eqref{eqn: thrust cone proj} is a special case of the projection formula of an icecream cone \cite[Exe. 29.12]{bauschke2017convex}.
\section{Numerical simulation and indoor flight experiments}
\label{sec: experiment}

We demonstrate the efficiency of Algorithm~\ref{alg: PIPG} by comparing its computation time against the state-of-the-art optimization solvers, and demonstrate the effectiveness of the trajectories computed by Algorithm~\ref{alg: PIPG} using indoor flight experiments via a custom quadrotor.

\subsection{Numerical simulation with randomly generated corridors}

We first evaluate the efficiency of the algorithms developed in Section~\ref{sec: trigger} and Section~\ref{sec: PIPG} using instances of optimization~\eqref{opt: nonconvex} with randomly generated corridors as follows. First, we let
\begin{equation}
    \overline{r}_0=\overline{v}_0=\overline{v}_f=\begin{bmatrix} 0 & 0 & 0\end{bmatrix}^\top, \enskip \overline{u}_f=-g.
\end{equation}
Second, we set the scalar parameters in optimization~\eqref{opt: nonconvex} using the values listed in Table~\ref{tab: quadrotor}. Third, we generate 100 random sequences of corridors, see Fig.~\ref{fig: dataset} for an illustration of the center lines of these corridor sequences. Each sequence contains 7 corridors. Each corridor starts at the origin and is uniquely characterized by four scalar parameters: radius, length, and two angles that defines its direction in a spherical coordinates--azimuthal angle and elevation angle. Each scalar parameter is sampled from a uniform distribution over an interval, see Table~\ref{tab: corridor} for the interval bounds of these parameters. Finally, we vary the number of corridors traversed by the trajectory by setting the final position \(\overline{r}_f\) to be the end point of different corridors in each sequence. 
\begin{figure}[!ht]
    \centering
    \includegraphics[width=0.45\columnwidth]{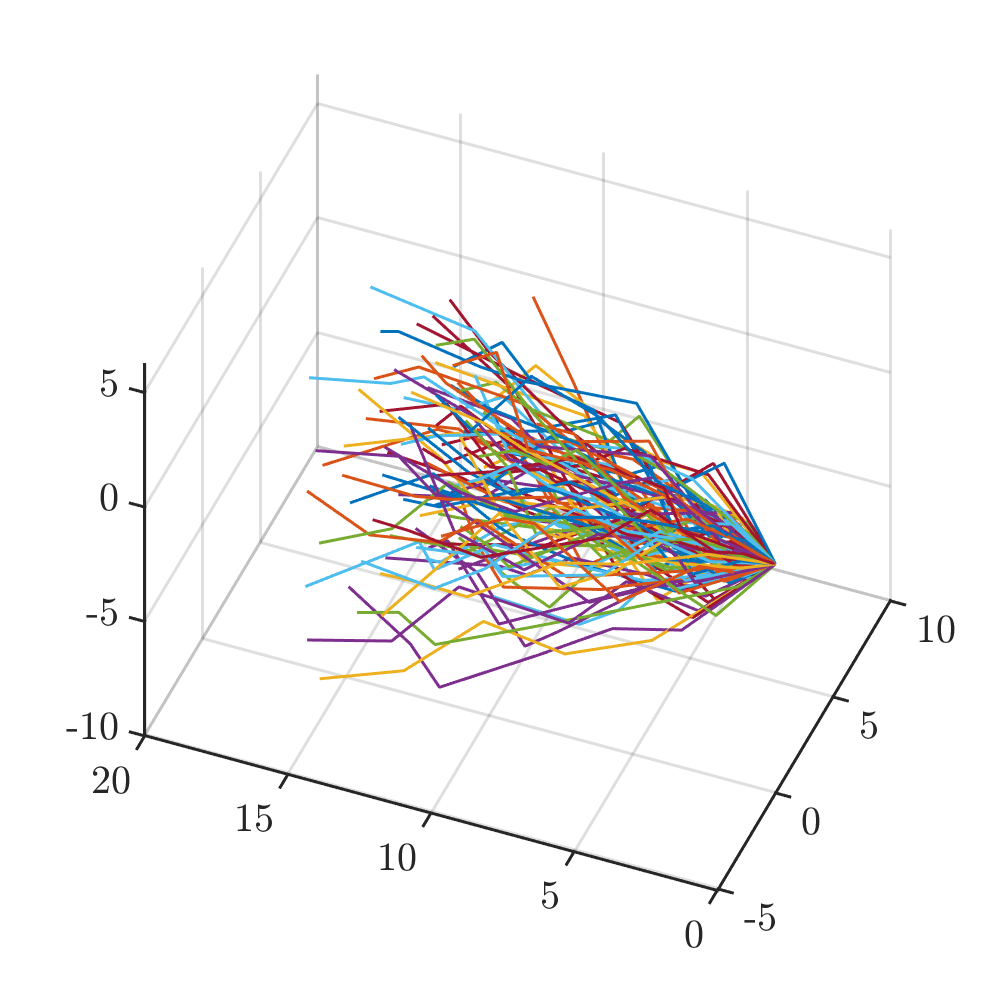}
    \caption{The center lines of the 100 random sequences of corridors. }
    \label{fig: dataset}
\end{figure}

\begin{table}[!ht]
    \centering
     \caption{The parameter values in optimization~\eqref{opt: nonconvex} (all units are omitted for simplicity)}
    \begin{tabular}{  c|c|c|c|c|c|c|c } 
 \hline
 \(m\) & \(\Delta\) & \(\omega\) & \(\xi\) & \(\underline{\gamma}\) &  \(\overline{\gamma}\) & \(\theta\) & \(\delta\)\\
 \hline \hline
 0.35 & 0.20 & 1.00 & 3.00 & 2.00 & 5.00 & \(\frac{\pi}{4}\) & 3.00\\ 
\hline
\end{tabular}
    \label{tab: quadrotor}
\end{table}

\begin{table}[!ht]
    \centering
     \caption{The interval bounds of the corridor parameters.}
    \begin{tabular}{  m{7em} | m{5em} } 
    \hline
 parameter & interval\\    
 \hline \hline
 radius & \([0.10, 0.50]\) \\

 length & \([1.00, 4.00]\) \\
 
 azimuthal angle & \([\frac{\pi}{4}, \frac{3\pi}{4}]\)\\

 elevation angle & \([-\frac{\pi}{4}, \frac{\pi}{4}]\)\\
 \hline
\end{tabular}
    \label{tab: corridor}
\end{table}

We demonstrate the performance of Algorithm~\ref{alg: trigger} using the aforementioned random instances of optimization~\eqref{opt: nonconvex}, where we use Algorithm~\ref{alg: PIPG} for infeasibility detection and optimizing a trajectory with time-varying corridor constraints. We implement the combination of Algorithm~\ref{alg: trigger} and Algorithm~\ref{alg: PIPG} in C++; see \url{https://github.com/Kartik-Nagpal/PIPG-Cpp} for details. We choose the values of time sequence \(\overline{\tau}_{[1, l]}\) and \(\underline{\tau}_{[1, l]}\) in Algorithm~\ref{alg: trigger} using the length of each corridor and the quadrotor's maximum speed, given by \(\xi\); and a coarse estimates of its minimum speed, given by \(\xi/2\).

Fig.~\ref{fig: time & cost} shows the computation time and solution quality of Algorithm~\ref{alg: trigger} combined with Algorithm~\ref{alg: PIPG}, and compares them against the performance of various combinations of Algorithm~\ref{alg: trigger}, mixed integer programming (MIP), off-the-shelf parser YALMIP \cite{lofberg2004yalmip}, commercial conic optimization solver GUROBI \cite{gurobi}, and open-source conic optimization solver ECOS \cite{domahidi2013ecos}. All numerical experiments are executed on a desktop computer equipped with the AMD Ryzen 9 5900X 12 Core Processor. Overall the combination of Algorithm~\ref{alg: trigger} and Algorithm~\ref{alg: PIPG} is about 50--200 times faster than the MIP approach as well as the combination of Algorithm~\ref{alg: trigger} and off-the-shelf solvers, at the price of at most a \(10\%\) increase in the cost function value. 

\begin{figure}[!ht]
\centering
  \begin{subfigure}{0.45\columnwidth}
  \includegraphics[trim=0.1cm 0.1cm 0.1cm 0.1cm,width=\textwidth]{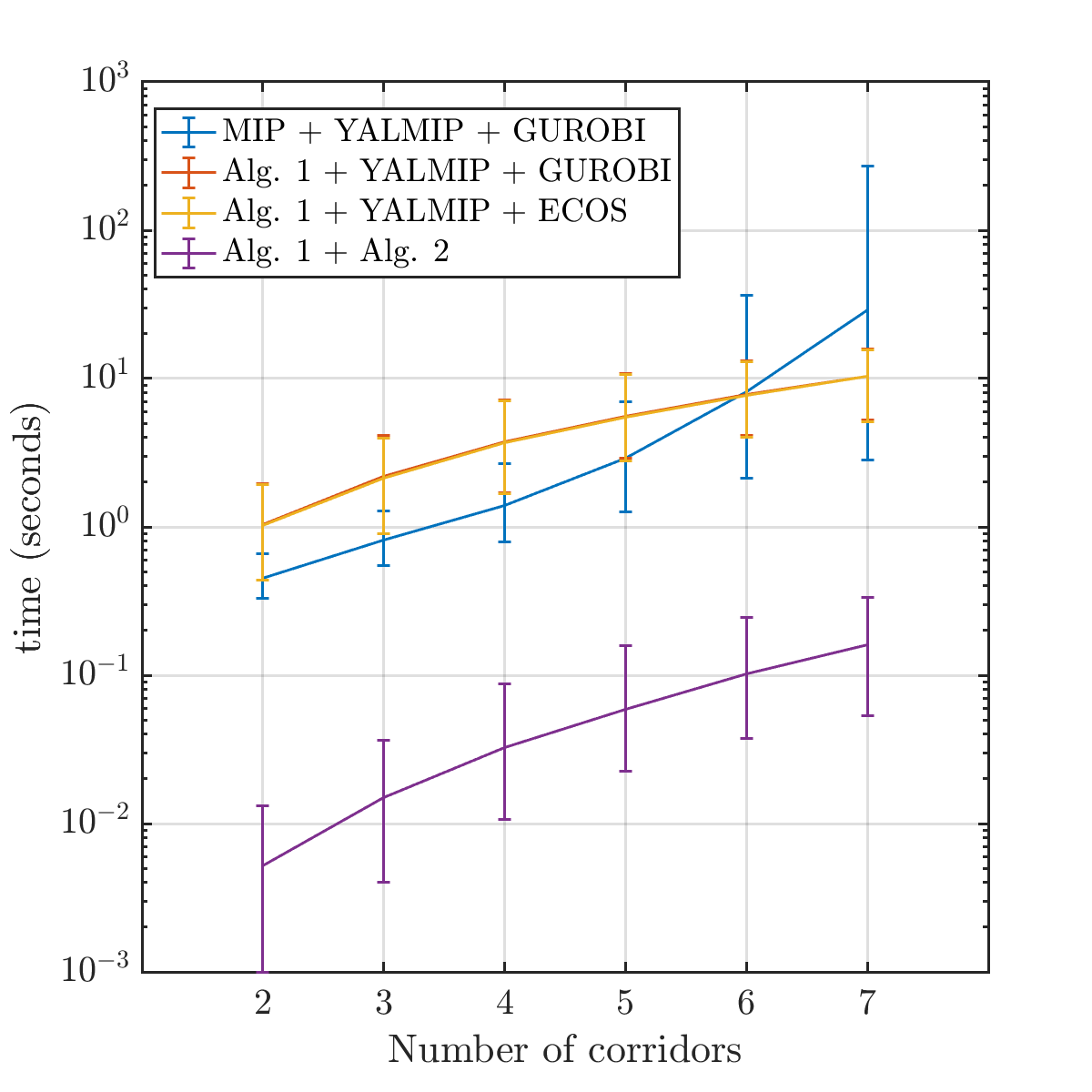}
  \caption{Computation time}\label{fig: time}
  \end{subfigure}
  \begin{subfigure}{0.45\columnwidth}
  \includegraphics[trim=0.1cm 0.1cm 0.1cm 0.1cm,width=\textwidth]{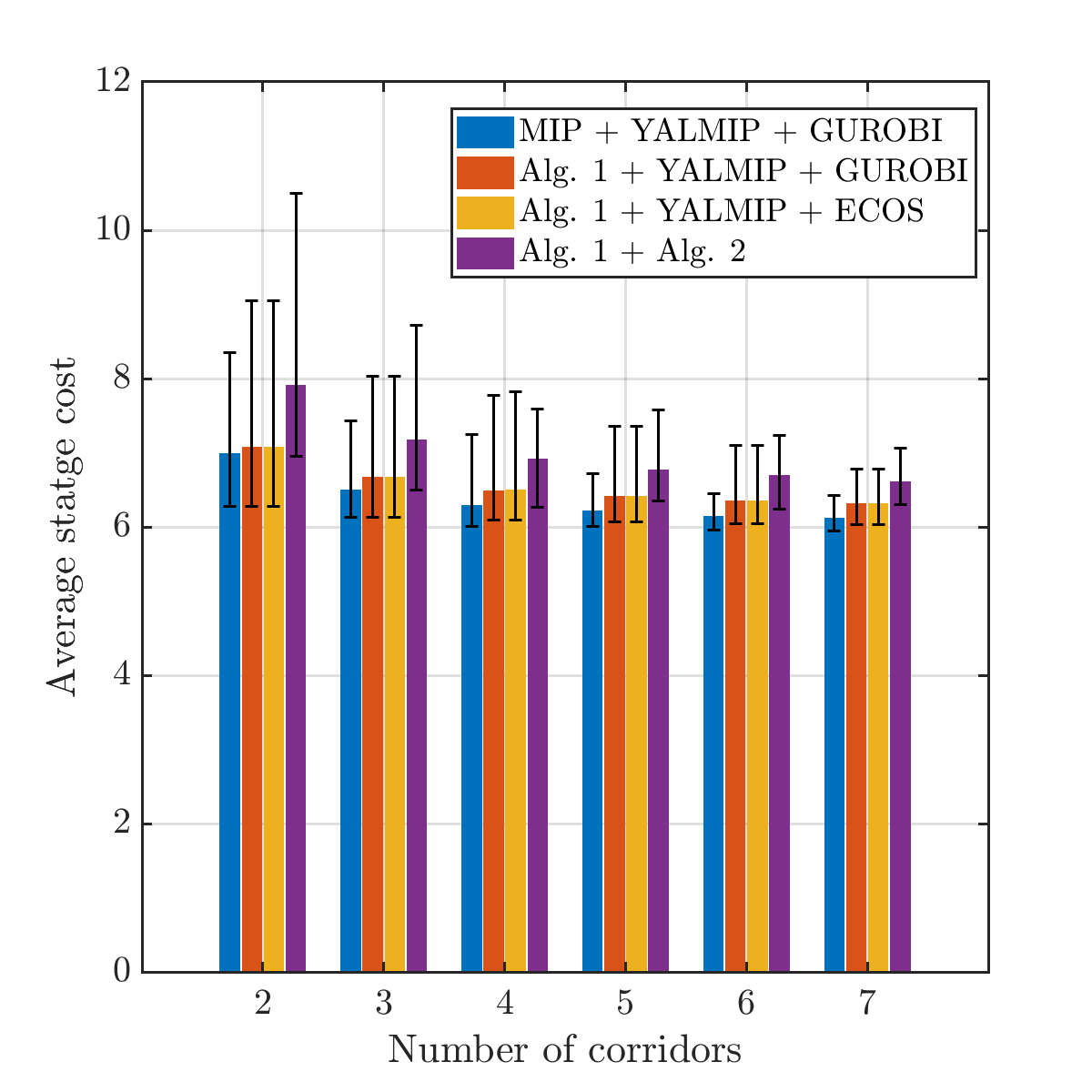}
  \caption{Trajectory cost divided by trajectory length.} \label{fig: C matrix}
  \end{subfigure} 
  \caption{The comparison of the computation time and the average state cost--which equals the objective function in optimization~\eqref{opt: nonconvex} divided by trajectory length \(t\)--of the trajectories computed by different solvers for optimization~\eqref{opt: nonconvex}, averaged over 100 randomly generated scenarios. The error bar shows the maximum and minimum value.  }
  \label{fig: time & cost}
\end{figure}

\subsection{Indoor flight experiments with hoop obstacles}

We demonstrate the application of Algorithm~\ref{alg: trigger} and Algorithm~\ref{alg: PIPG} using the quadrotor platform in the Autonomous Control Laboratory ( see \url{https://depts.washington.edu/uwacl/}). This platform contains a custom-made quadrotor equipped with a  2200-milliAmp-hour lithium-polymer battery; accelerometers and gyroscopes that measure the acceleration and the angular velocity, respectively, at a 100-1000 Hz rate; a 500 MHz dual-core Intel Edison and a 1.7 GHz quad-core Intel Joule processor; and an IEEE 802.11n compliant WiFi communication link. See Fig.~\ref{fig: acl quad} for an illustration. The platform also include an 4 meters by 7 meters by 3 meters indoor flight space, equipped with an OptiTrack motion capture system that can measure the attitude and position of a quadrotor at 50-150 Hz rate. 

We conduct the quadrotor flight experiments using the trajectories computed by Algorithm~\ref{alg: trigger} and Algorithm~\ref{alg: PIPG} as reference guidance. We also use hoop obstacles to mark out the boundary of each flight corridor. Fig.~\ref{fig: flight exp} shows the reference trajectories and experiment trajectories in three different corridor scenarios\footnote{To ensure flight safety, we use a reduced hoop radius (about 20\% of the actual size) when computing the flight trajectories.}. These experiments demonstrate how to use the proposed approach in actual flight experiments in cluttered environments.

\begin{figure}[!ht]
    \centering
    \input{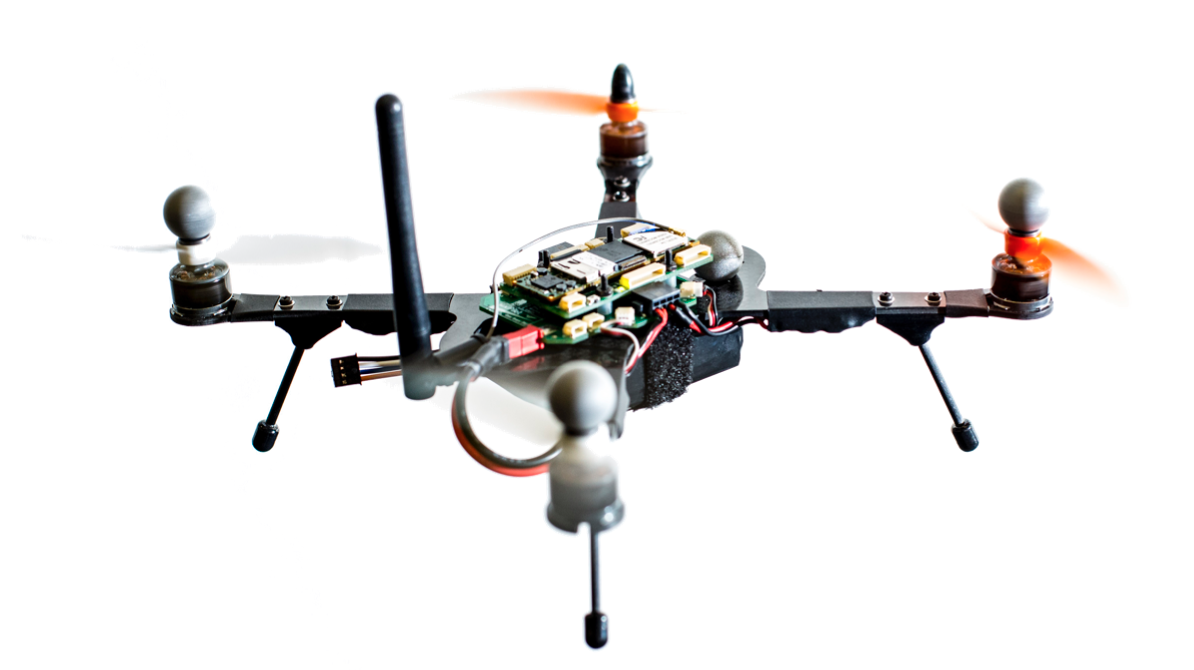}
    \caption{The custom quadrotor in the Autonomous Control Laboratory.}
    \label{fig: acl quad}
\end{figure}

\begin{figure}
    \centering
    \includegraphics[width=0.4\textwidth]{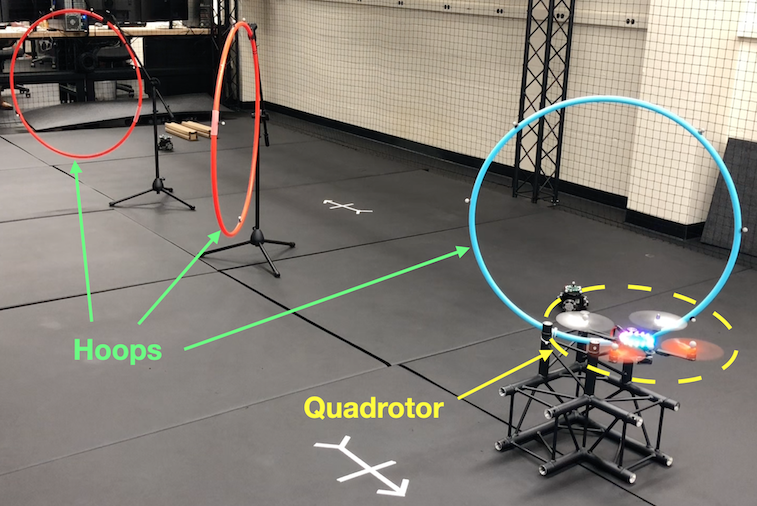}
    \caption{The indoor flight environment with hoop obstacles.}
    \label{fig: acl lab}
\end{figure}

\begin{figure}[!ht]
\centering
  \begin{subfigure}{0.49\columnwidth}
  \includegraphics[trim=0.1cm 0.1cm 0.1cm 0.1cm,width=\textwidth]{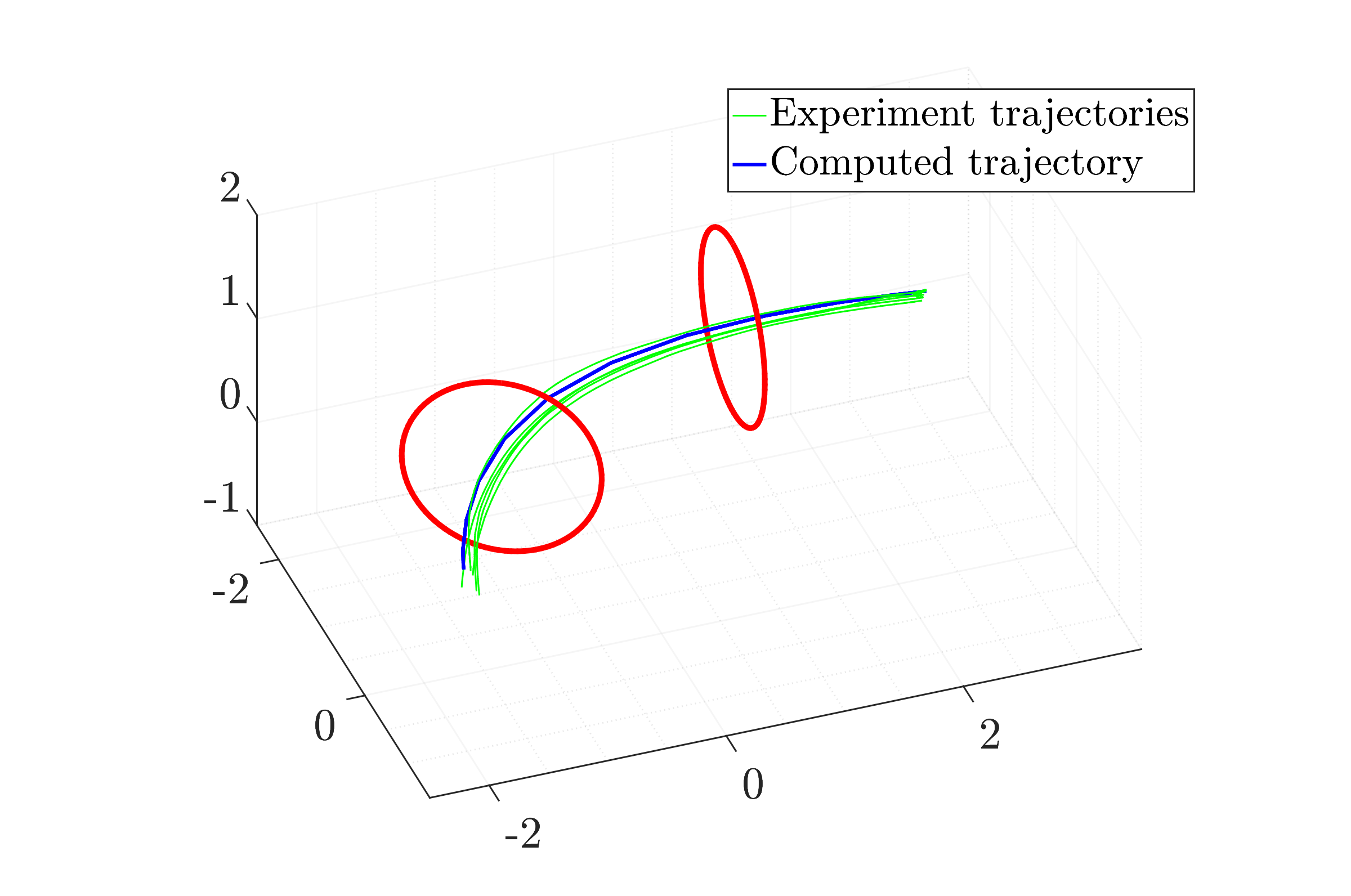}
  \caption{Scenario 1.}
  \end{subfigure}
  \begin{subfigure}{0.49\columnwidth}
  \includegraphics[trim=0.1cm 0.1cm 0.1cm 0.1cm,width=\textwidth]{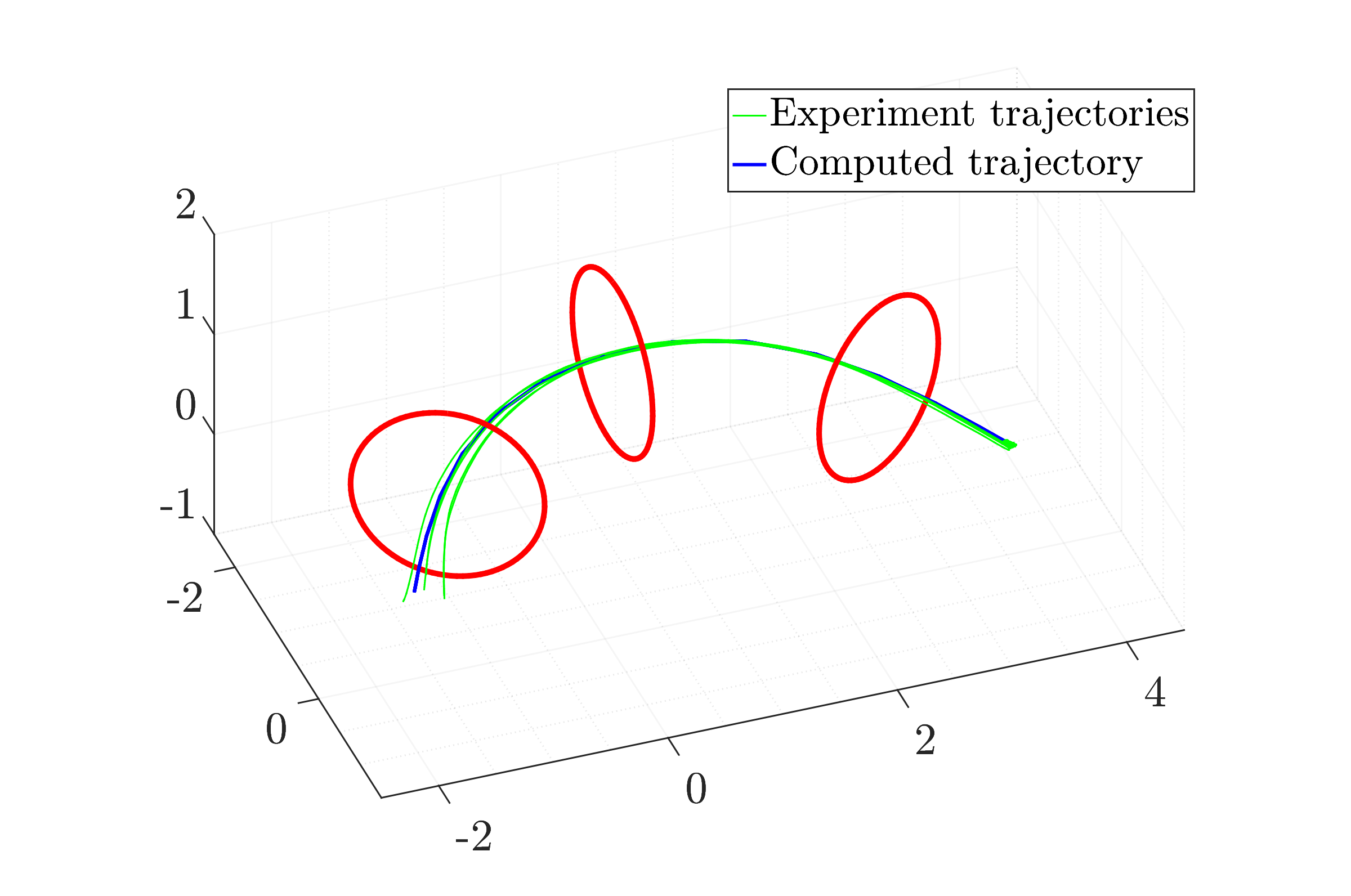}
  \caption{Scenario 2.}
  \end{subfigure} 
  \begin{subfigure}{0.49\columnwidth}
  \includegraphics[trim=0.1cm 0.1cm 0.1cm 0.1cm,width=\textwidth]{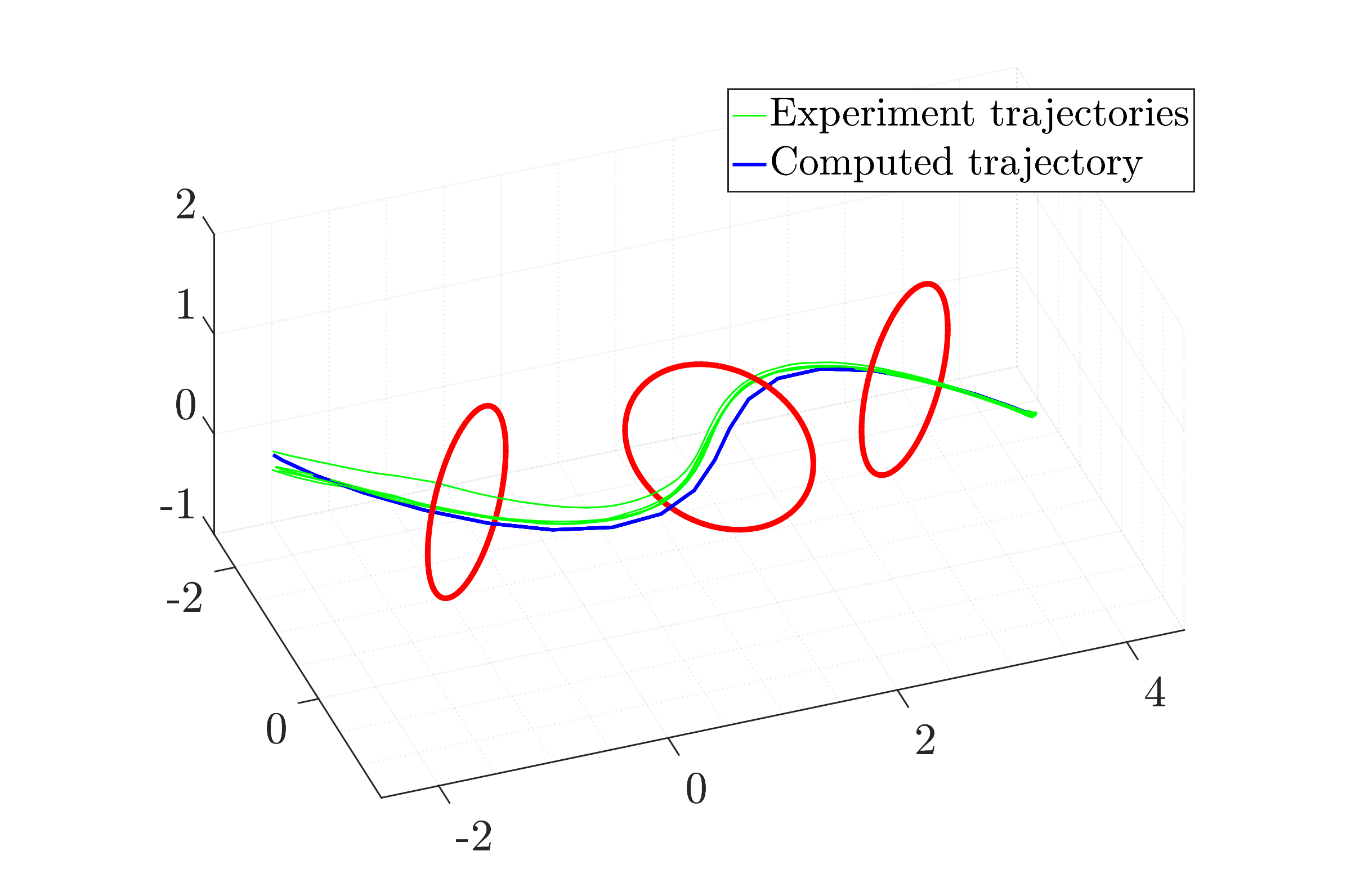}
  \caption{Scenario 3.}
  \end{subfigure} 
  \caption{Reference trajectory computed by Algorithm~\ref{alg: trigger} and Algorithm~\ref{alg: PIPG} and the measured flight trajectories in experiments. For each scenario, we showcase the measured trajectories in five separate flight experiments.}
  \label{fig: flight exp}
\end{figure}

\section{Conclusion}
\label{sec: conclusion}

We introduce a novel bisection method that approximates the nonconvex corridor constraints using time-triggered convex corridor constraints, and develop customized implementation of this method that enables real-time trajectory optimization subject to second-order cone constraints. Our results provide a novel benchmark solution approach for trajectory optimization, which is about 50--200 times faster than mixed integer programming in numerical experiments. Future direction includes onboard implementation and extensions to trajectory optimization with nonlinear dynamics model, such as six-degree-of-freedom rigid body dynamics for space vehicles \cite{malyuta2021advances}.



\bibliography{reference}

\end{document}